\documentclass[preprint,12pt]{elsarticle}

\usepackage{amssymb}
\usepackage{graphics}
\usepackage{amsmath}
\usepackage{amssymb}
\usepackage{amsthm}
\usepackage{xcolor}
\usepackage{enumerate}
\usepackage{subfig}
\usepackage{mathrsfs} 
\usepackage{amsthm}

\theoremstyle{definition}
\newtheorem{definition}{Definition}

\theoremstyle{theorem}
\newtheorem{theorem}{Theorem}

\theoremstyle{lemma}
\newtheorem{lemma}{Lemma}

\theoremstyle{example}
\newtheorem{example}{Example}

\theoremstyle{remark}
\newtheorem{remark}{Remark}

\theoremstyle{prop}
\newtheorem{prop}{Proposition}

\begin{document}

\begin{frontmatter}



\title{A new Representation of $\alpha$-Bernstein Operators }


\author{Jamshid  Saeidian \corref{cor1}}
\ead{j.saeidian@khu.ac.ir}
\author{ Bahareh Nouri}

 \address{Faculty of Mathematical Sciences and Computer, Kharazmi University, No. 50,  Taleghani ave., Tehran, Iran.}

\cortext[cor1]{Corresponding author}
%

\begin{abstract}

The $\alpha$-Bernstein operators were initially introduced in the paper by Chen, X., Tan, J., Liu, Z., Xie, J. (2017) titled "Approximation of Functions by a New Family of Generalized Bernstein Operators" (Journal of Mathematical Analysis and Applications, 450(1), 244-261). Since their introduction, these operators have served as a source of inspiration for numerous research endeavors. In this study,  we propose a novel  technique, founded on a recursive relation,  for constructing Bernstein-like bases. A special case of this new representation yields a novel portrayal of Chen's operators. This innovative representation enables the discovery of additional properties of $\alpha$-Bernstein operators and facilitates alternative and more straightforward proofs for certain theorems.
\end{abstract}



\begin{keyword}
Bernstein-like operator \sep  Generalized operators \sep Positive linear operator \sep Shape-preserving approximation \sep Voronovskaja-type theorem

\MSC[2020]  41A25 \sep 26A15 \sep  47A58 \sep 41A10

\end{keyword}

\end{frontmatter}


\section{Introduction}  \label{}

The Bernstein polynomials, rooted in the fertile ground of approximation theory, stand as an enduring pillar of mathematical research and practice. These polynomials, introduced by S.N. Bernstein in 1912, not only define a fundamental class of algebraic polynomials but also wield a profound influence in mathematical analysis. It was through the lens of Bernstein polynomials that S.N. Bernstein provided the first constructive proof of Weierstrass' approximation theorem, a seminal contribution that continues to reverberate through the annals of mathematical history \cite{farouki2012bernstein}.

Defined as:
\[B_{n,k}(z) ={n \choose k} z^k(1-z)^{n-k}, ~~k=0,...,n ~~ \& ~~ n\in \mathbb{N},\]
these polynomials are instrumental in approximating functions over the interval \([0,1]\). The core of this approximation is encapsulated in  Bernstein operators with the formula:
\[ \mathcal{B}_n(f;z) =\sum_{k=0}^n B_{n,k}(z)f(\frac{k}{n}),  ~~~ z\in [0,1],\]
where a weighted sum of function values at equidistant points on the interval offers a versatile tool for approximation.

In 2017, Chen et al. \cite{chen2017approximation} proposed an innovative modification of Bernstein operators, centered around the concept of $\alpha$-Bernstein polynomials as follow:
\begin{definition}\cite{chen2017approximation}
The $ \alpha $-Bernstein polynomials  of degree $n$, $B_{n,k}^{(\alpha)}(z)$, are defined by $B_{1,0}^{(\alpha)}(z)=1-z$, $B_{1,1}^{(\alpha)}(z)=z$ and
\begin{equation} 
B_{n,k}^{(\alpha)}(z) = \left[ \binom{n-2}{k}(1-\alpha)z + \binom{n-2}{k-2}(1-\alpha)\left(1-z\right) + \binom{n}{k}\alpha  z(1-z) \right]z^{k-1}(1-z)^{n-k-1}   \label{chenformula}
\end{equation}
where $n \geq 2$, $z \in[0, 1]$ and $\binom{m}{l}$ denotes  the binomial coefficients.
\end{definition}

Then, the corresponding $ \alpha $-Bernstein operators are defined to read
\[ \mathcal{B}_n^{\alpha}(f;z) =\sum_{k=0}^n B_{n,k}^{\alpha}(z)f(\frac{k}{n}),  ~~~ z\in [0,1],\] 
for any $f \in  C[0,1]$, i.e. continuous functions.

In recent years, the field of approximation theory has experienced noteworthy advancements, propelled in part by the pioneering work of Chen and his colleagues. Their contributions to the theory of Bernstein-like operators have not only broadened our comprehension but also acted as a catalyst for subsequent exploration and refinement. This seminal work has sparked significant interest and inspired further investigations into various families of Bernstein-like operators \cite{mohiuddine2017construction, kajla2018blending, mohiuddine2020approximation, ozger2020approximation}.


The $\alpha$-Bernstein operators are commonly identified as generalized Bernstein operators in the literature, and researchers have further developed the underlying concept of this modification to generalize other prominent operators in the field.
Among these extensions, we can mention operators linked to Schurer polynomials, Kantorovich polynomials, Stancu polynomials, $q$-Bernstein polynomials, Durrmeyer polynomials, as well as other operators like Favard-Szász-Mirakjan and Baskakov, these  extensions, as well as other related works could be traced in \cite{    mohiuddine2017construction, ozger2020approximation,  mohiuddine2021approximation, aral2019parametric, cheng2023construction, kajla2019generalized}.

Inspired by Chen's research, this paper presents a novel framework for generating Bernstein-like basis functions. This innovative structure can be viewed as an extension of Bernstein polynomials, aligning with these classical polynomials and their $\alpha$-Bernstein modifications  \cite{chen2017approximation} in specific scenarios. This framework enables us to  introduce a recursive  representation of the  $\alpha$-Bernstein operators, a representation that promises to enhance our comprehension of its underlying principles and applications. By revisiting the operator's structure through a fresh lens, we aim to offer a unique perspective on its mathematical properties and computational implications.


Our primary focus lies in presenting this new representation and demonstrating its utility through  alternative  proofs for selected theorems associated with $\alpha$-Bernstein operators.
Moreover, it provides us with the opportunity to uncover some additional properties of $\alpha$-Bernstein operators.


The structure of the rest of this paper is as follows: Section 2 delineates the novel representation and validates its equivalence with Chen's formula. Section 3 examines the properties of this new presentation within the context of the Bernstein-like bases. In Section 4, on one hand, two novel properties of the $\alpha$-Bernstein operators are derived, and subsequently, we provide alternative proofs for key theorems, illustrating the practical implications of our proposed representation. Finally, Section 5 presents concluding remarks and proposes potential avenues for future research.

\section{A general structure for constructing Bernstein-like bases}

The idea of recursively constructing Bernstein-like basis functions has already been employed by researchers,
 Yan and Liang \cite{yan2011extension} utilized this approach to construct a Bezier-like model and  Li \cite{li2018novel} extended the idea to define a parameter-based Bernstein-like basis, suitable for constructing Bezier-type curves. Additionally, Bibi et al. \cite{bibi2020geometric} introduced the so-called \textit{generalized hybrid trigonometric Bernstein basis}, and Ameer and his colleagues \cite{ameer2022novel} developed generalized Bernstein-like basis functions. Both of these studies further contribute to the recursive construction of Bernstein-like bases.

The idea is straightforward: identify an appropriate set of initial bases, such as $\{F_{2,0},F_{2,1},F_{2,2}\}$, and generate higher-order bases through the established recursive relation of Bernstein polynomials \cite{farin2002curves}:
  \begin{equation}
F_{n,i}(z) = (1-z)F_{n-1,i}(z) + zF_{n-1,i-1}(z) \quad \text{for } z \in [0,1], \quad i = 0, 1, 2, \ldots, n. \label{rec-rel}
\end{equation}

The count of initial bases, hereafter referred to as the \textit{starting basis}, can be any natural number. Yan and Liang \cite{yan2011extension} and Ameer et al. \cite{ameer2022novel} both employed three polynomials of degree 4, Li \cite{li2018novel} selected four polynomials of degree 3 as the starting basis, and Bibi et al. \cite{bibi2020geometric} utilized three hybrid trigonometric functions. It is essential for the starting basis functions to exhibit as many fundamental properties of Bernstein bases as possible.

We stick to this technique and propose  a general structure to define a family of Bernstein-like bases.

\begin{definition} \label{defphi}
The starting Bernstein-like basis functions of order two are  defined  for $z \in [0,1]$ as 
\begin{eqnarray} 
g_{2,0}(z)&=&az+b+\varphi(z),\nonumber \\
g_{2,1}(z)&=&1-2b-a-2\varphi(z), \label{formula-phi}\\
g_{2,2}(z)&=&-az+b+a+\varphi(z),\nonumber
\end{eqnarray} 
where $a,b \in \mathbb{R}$  and the real valued function $\varphi$ must satisfy some conditions. The higher order Bernstein-like bases are  generated by recursive relation (\ref{rec-rel}).
\end{definition}

\begin{lemma}\label{lem1} 
 The Bernstein-like basis functions of order \(n \geq 2\), as defined in Definition \ref{defphi}, satisfy the partition of unity property, expressed as \( \sum_{i=0}^{n}g_{n,i}(z)=1 \). Additionally, if the function \(\varphi\) is constrained to satisfy \(\varphi(1-z)=\varphi(z)\), then a symmetry feature emerges: \( g_{n,i}(z)=g_{n,n-i}(1-z) \) for \( i = 0, 1, 2, \ldots , n \).

\end{lemma}

One can impose further restrictions on the  function $\varphi$ and the real values $a$ and $b$ in order to satisfy more fundamental properties of classical Bernstein polynomials. In fact,  the non-negativity and end-point coincidence with Bernstein bases are two unavoidable features.

\begin{example}
When  we set $a=-1$, $b=0$ and $\varphi(z)=z^2-z+1$, the  polynomials in equation (\ref{formula-phi}) reduce to the   Bernstein polynomials of degree 2, thus Definition \ref{defphi} constructs the classical Bernstein polynomials.
\end{example}

\begin{example}
Authors in \cite{nouri2022class} presented a special case of (\ref{formula-phi}) where they used the  function $\varphi (z)=\sqrt{\left ( 1-\nu  \right )\left ( z^{2}-z \right )+\frac{1}{4}}$, along with $a=-1$ and $b=\frac{1}{2}$. 
It is verified that the so-called $sq$-basis functions also satisfy the non-negativity and end-point interpolation properties. Furthermore, the corresponding operators exhibit the most common features of the classical Bernstein operators \cite{nouri2023family}.
\end{example}

\section{New representation of the $\alpha$-Bernstein polynomials}

Setting $a=-1$, $b=1$  and $\varphi(z)=\alpha (z^2-z)$, in equation (\ref{formula-phi}), leads us to a set of starting basis functions as follows:
\begin{eqnarray} 
F_{2,0}(z)&=&1-z+\alpha \left ( z^{2}-z \right ),\nonumber \\
F_{2,1}(z)&=&-2\alpha \left ( z^{2}-z \right ), \label{formula1}\\
F_{2,2}(z)&=&z+\alpha \left ( z^{2}-z \right ).\nonumber
\end{eqnarray} 
Now employing the recurrence relation (\ref{rec-rel}), one constructs a set of basis functions which its equivalence with $\alpha$-Bernstein polynomials of Chen \cite{chen2017approximation} are easily verified, thus we have:
\begin{theorem}

The procedure presented in Definition \ref{defphi}, along with starting basis functions (\ref{formula1}), generates the $\alpha$-Bernstein polynomials.
\end{theorem}

The recursive definition of the $\alpha$-Bernstein polynomials provides an opportunity to derive practical computational formulas, one of which is articulated in the following lemma.

\begin{lemma}\label{lemsigma}
The  $\alpha $-Bernstein polynomials, defined by equations (\ref{rec-rel}) and (\ref{formula1}), can be derived through the following relation:
\begin{eqnarray}\label{sigma-eq}
F_{n,i}(z) =\sum_{j=0}^{n-2}B_{n-2,j}(z)F_{2,i-j}(z),  ~~i=0,1,\cdots,n,
\end{eqnarray}
where  $ B_{n-2,j}  $  are the classical Bernstein polynomials of order $n-2$.
\end{lemma}
\begin{proof}
We utilize the method of induction to establish the validity of the result, initiating with the base case  $n = 3$:
\begin{eqnarray*}
F_{n,0}(z) &=&\sum_{j=0}^{1}B_{1,j}(z)F_{2,-j}(z)=B_{1,0}(z)F_{2,0}(z)=(1-z)F_{2,0}(z),\\
F_{n,1}(z) &=&\sum_{j=0}^{1}B_{1,j}(z)F_{2,1-j}(z)=B_{1,0}(z)F_{2,1}(z)+B_{1,1}(z)F_{2,0}(z)=(1-z)F_{2,1}(z)+zF_{2,0}(z),\\
F_{n,2}(z) &=&\sum_{j=0}^{1}B_{1,j}(z)F_{2,2-j}(z)=B_{1,0}(z)F_{2,2}(z)+B_{1,1}(z)F_{2,1}(z)=(1-z)F_{2,2}(z)+zF_{2,1}(z),\\
F_{n,3}(z) &=&\sum_{j=0}^{1}B_{1,j}(z)F_{2,3-j}(z)=B_{1,1}(z)F_{2,2}(z)=zF_{2,2}(z),\\
\end{eqnarray*}
the desired outcome is achieved.

Now, assuming the validity of the relation for $ n-1 $, we obtain:
\begin{eqnarray*}
F_{n-1,i}(z) =\sum_{j=0}^{n-3}B_{n-3,j}(z)F_{2,i-j}(z), ~~ i=0,1,\cdots,n-1,
\end{eqnarray*}
In the next step, we aim to establish the validity of the relation for  $n$. Utilizing the recursive relation (\ref{rec-rel}), for  $i=0,1,\cdots,n$, we have:
\begin{eqnarray*}
F_{n,i}(z)& =& (1-z)F_{n-1,i}(z) + zF_{n-1,i-1}(z)\\
&=&(1-z)\sum_{j=0}^{n-3}B_{n-3,j}(z)F_{2,i-j}(z)+ z\sum_{j=0}^{n-3}B_{n-3,j}(z)F_{2,i-1-j}(z)\\
&=&(1-z)B_{n-3,0}(z)F_{2,i-0}(z)+\sum_{j=1}^{n-3}(1-z) B_{n-3,j}(z)F_{2,i-j}(z)\\
&&~~~~~~~+ \sum_{j=1}^{n-3}zB_{n-3,j-1}(z)F_{2,i-j}(z)+zB_{n-3,n-3}(z)F_{2,i-(n-2)}(z)\\
&=&(1-z)B_{n-3,0}(z)F_{2,i-0}(z)+\sum_{j=1}^{n-3}[(1-z) B_{n-3,j}(z)+zB_{n-3,j-1}(z) ]F_{2,i-j}(z)\\
  &&    ~~~~~~~+zB_{n-3,n-3}(z)F_{2,i-(n-2)}(z)\\
&=&B_{n-2,0}(z)F_{2,i-0}(z)+\sum_{j=1}^{n-3}B_{n-2,j}(z)F_{2,i-j}(z)+B_{n-2,n-2}(z)F_{2,i-(n-2)}(z)\\
&=&\sum_{j=0}^{n-2}B_{n-2,j}(z)F_{2,i-j}(z).
\end{eqnarray*}

\end{proof}

\begin{remark}
For any given value of \(n\), the right-hand side of (\ref{sigma-eq}) has at most  three non-zero terms, providing a computational advantage. Moreover, relation (\ref{sigma-eq}) can be represented in a matrix form as follows:
\begin{eqnarray*}
 \begin{bmatrix}
F_{n,0}(z)\\ 
F_{n,1}(z)\\ 
\vdots \\
F_{n,n}(z)
\end{bmatrix}=
\begin{bmatrix}
B_{n-2,0}(z) & 0 &0 \\ 
B_{n-2,1}(z) & B_{n-2,0}(z) & 0 \\ 
B_{n-2,2}(z) & B_{n-2,1}(z) & B_{n-2,0}(z) \\ 
B_{n-2,3}(z) & B_{n-2,2}(z) & B_{n-2,1}(z)  \\ 
\vdots  &  \vdots & \vdots \\ 
B_{n-2,n-3}(z) & B_{n-2,n-4}(z)  & B_{n-2,n-5}(z) \\ 
B_{n-2,n-2}(z) & B_{n-2,n-3}(z)  & B_{n-2,n-3}(z) \\ 
0 & B_{n-2,n-2}(z) & B_{n-2,n-3}(z) \\ 
 0& 0 & B_{n-2,n-2}(z)
\end{bmatrix}
\begin{bmatrix}
F_{2,0}(z)\\ 
F_{2,1}(z)\\ 
F_{2,2}(z)
\end{bmatrix}.
\end{eqnarray*}
%
\end{remark}

From the proof of Lemma \ref{lemsigma}, the following fact is straightforward:

\begin{lemma} \label{rec-lem}
For any set of basis functions $\{F_{n,i}\}_{i=0}^n$, which are constructed using the recursive relation (\ref{rec-rel}), any basis of order $n$ could be represented in terms of the basis elements of order $m <n$, i.e.
\begin{eqnarray} \label{sigma}
F_{n,i}(z) =\sum_{j=0}^{n-m}B_{n-m,j}(z)F_{m,i-j}(z), ~~~ i=0,\cdots, n.
\end{eqnarray}

\end{lemma}

\begin{remark}
 Lemma \ref{rec-lem} provides a universal formula for expressing the basis functions in terms of  the starting basis functions, irrespective of the number of starting bases employed.
\end{remark}

The forthcoming result demonstrates that the $\alpha$-Bernstein polynomials possess many of the properties of classical Bernstein polynomials, making them well-suited for function approximation.

\begin{theorem}\label{thm1} 
The $\alpha$-Bernstein polynomials  derived from equations  (\ref{rec-rel}) and (\ref{formula1}) exhibit the following  properties:
\begin{itemize}
\item [\rm{(a)}] Nonnegativity: $ F_{n,i}(z) \geq 0$  for $ i = 0, 1, 2,\ldots , n$.
\item [\rm{(b)}] Partition of unity: $ \sum_{i=0}^{n}F_{n,i}(z)=1 $.
\item [\rm{(c)}] Symmetry: $ F_{n,i}(z)=F_{n,n-i}(1-z) $ for $ i = 0, 1, 2, \ldots , n $. 
\item [\rm{(d)}] End-point values:
\begin{eqnarray}
F_{n,i}(0)= 
\begin{cases} 
1, & i=0,
\\ 
0, & i\neq 0,
\end{cases}
\hspace{0.7cm}
F_{n,i}(1)= 
\begin{cases} 
1, & i=n,
\\ 
0, & i\neq n.
\end{cases}
\label{end-val}
\end{eqnarray}
\end{itemize}
\end{theorem}

\begin{proof}
The proof of this theorem relies on the application of the  mathematical induction.
\begin{itemize}
\item [\rm{(a)}]
Non-negativity is obvious by  (\ref{rec-rel}) and  (\ref{formula1}).
\item [\rm{(b)}]
It is  already stated in Lemma \ref{lem1}.
\item [\rm{(c)}]
The starting basis functions exhibit symmetry. Let's assume that the $\alpha$-Bernstein polynomials of order $k$  hold  symmetry. Now, by considering this inductive hypothesis along with the recurrence relation (\ref{rec-rel}), we can conclude that
\begin{eqnarray*}
F_{k+1,i}(1-z)&=&(1-(1-z))F_{k,i}(1-z)+(1-z)F_{k,i-1}(1-z)\\
&=&zF_{k,k-i}(z)+(1-z)F_{k,k-i+1}(z)=F_{k+1,k-i+1}(z).
\end{eqnarray*}
\item [\rm{(d)}]
Based on a straightforward deduction from (\ref{formula1}), it can be concluded that the results in (\ref{end-val}) apply to the case when $n=2$.
\\
Let's assume that the properties at the endpoints are satisfied by the $\alpha$-Bernstein polynomials of order $k$. Consequently, by utilizing the inductive hypothesis and equation (\ref{rec-rel}), we can conclude that:
\begin{eqnarray}
F_{k+1,i}(0)&=&F_{k,i}(0)= 
\begin{cases} 
1 & i=0,
\\ 
0, & i\neq 0,
\end{cases} \nonumber
\\ 
F_{k+1,i}(1)&=&F_{k,i-1}(1)= 
\begin{cases} 
1 & i=k+1,
\\ 
0, & i\neq k+1.
\end{cases} \nonumber
\end{eqnarray}

\end{itemize}
\end{proof}
In Figure \ref{fig1}, an illustration of the $\alpha$-Bernstein polynomials  is presented. These functions are generated for varying values of $ \alpha $ (including $ 0,0.2, 0.4,0.6,0.8$ and $ 1 $), while the values of $ n $ are set to  $2, 3,$ and $4$.

\begin{figure}[ht]
\centering
\subfloat 
{\includegraphics[width=6.5cm]{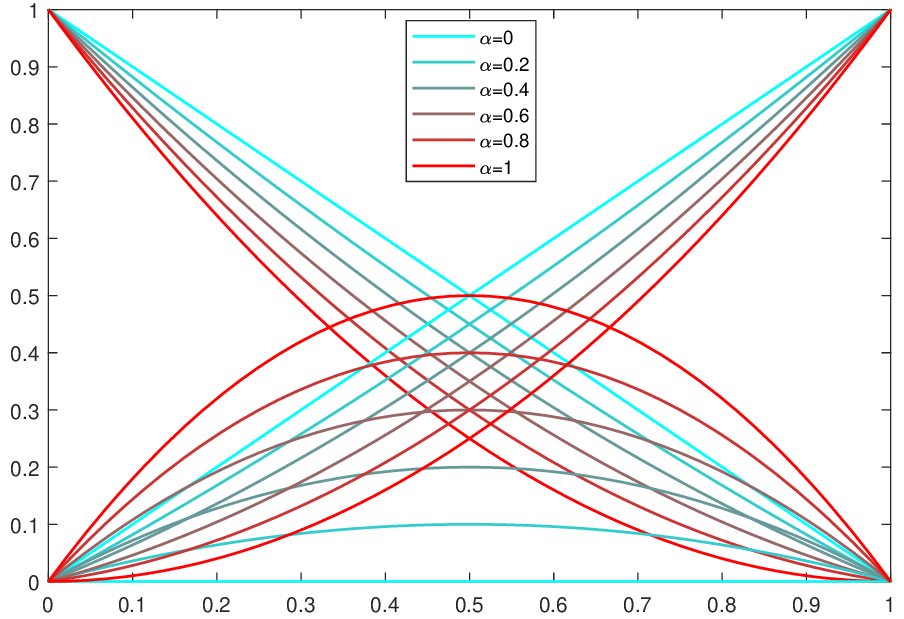}}
\\
\subfloat
{\includegraphics[width=6.5cm]{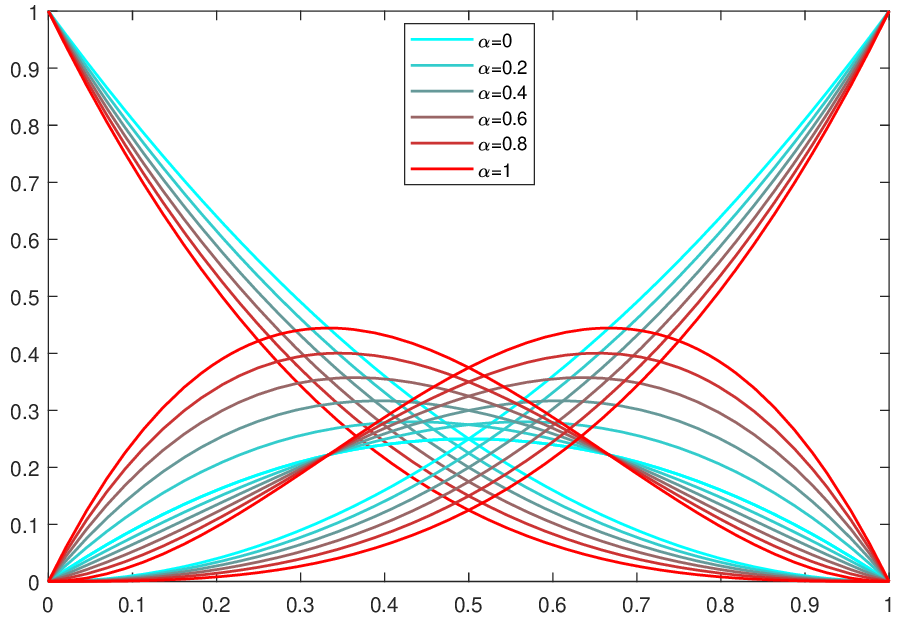}}
\hspace*{0.1mm}
\subfloat
{\includegraphics[width=6.5cm]{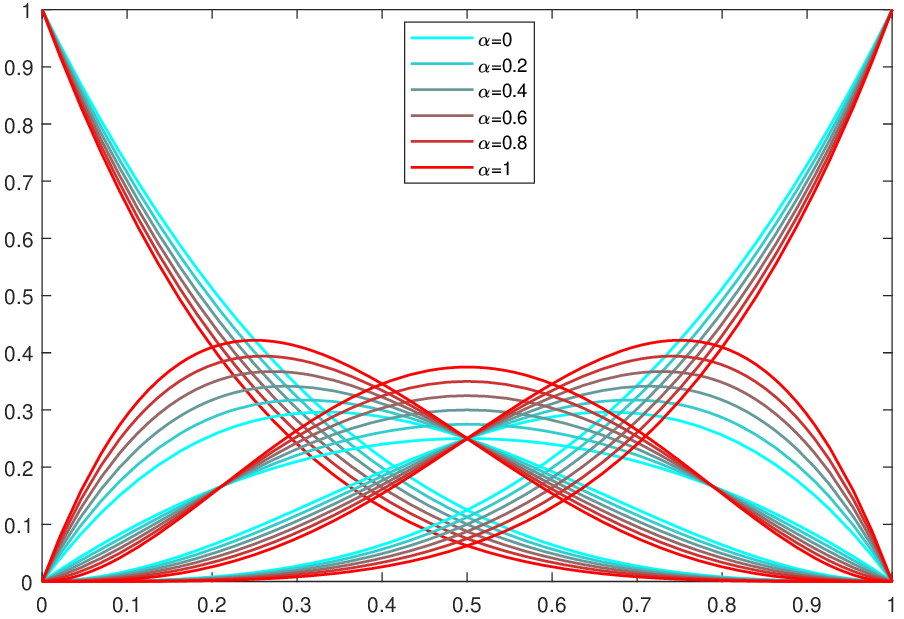}}
\caption{The $\alpha$-Bernstein polynomials  by adjusting $ \alpha $ from  $ 0, 0.2,\ldots,1$.} 
\label{fig1}
\end{figure}

\section{New results and alternative proofs}


This section is dedicated to showcasing the benefits of the new representation of $\alpha$-Bernstein operators. We introduce a new formula for computing the moments and present an analogue of Vornovskaja's theorem. Subsequently, we verify the shape-preserving properties of these operators using entirely different strategies.

\begin{theorem}
For all $j \in \mathbb{N}\cup \left \{ 0 \right \}, n \in \mathbb{N}$, $\alpha \in [0,1]$ and $  z \in \left [ 0,1 \right ] $, we have the recurrence
formula:
\begin{eqnarray*}
\mathcal{B}_n^{\alpha}\left ( z^{j+1};z \right )=\left ( 1-\frac{1}{n} \right )^{j+1}\left [ \left ( 1-z \right )\mathcal{B}_{n-1}^{\alpha}\left ( z^{j+1};z \right )+z\sum_{k=0}^{j+1} \binom{j+1}{k}\dfrac{1}{\left ( n-1 \right )^{k}} \mathcal{B}_{n-1}^{\alpha}\left ( z^{j+1-k};z \right )\right ].
\end{eqnarray*}
\end{theorem}
\begin{proof}
Employing equation (\ref{rec-rel})  we derive:
\begin{eqnarray*}
\mathcal{B}_n^{\alpha}\left ( z^{j+1};z \right )&=& \sum_{i=0}^{n}\left ( \frac{i}{n} \right )^{j+1}F_{n,i}(z)\\
&=&\sum_{i=0}^{n}\left ( \frac{i}{n} \right )^{j+1}\left [ \left ( 1-z \right )F_{n-1,i}(z)+zF_{n-1,i-1}(z) \right ]\\
&=&\frac{\left ( n-1 \right )^{j+1}}{n^{j+1}}\left [ \left ( 1-z \right )\sum_{i=0}^{n}\left ( \frac{i}{n-1} \right )^{j+1}F_{n-1,i}(z)+z\sum_{i=0}^{n}\left ( \frac{i}{n-1} \right )^{j+1}F_{n-1,i-1}(z) \right ]\\
&=&\left ( 1-\frac{1}{n} \right )^{j+1}\left [ \left ( 1-z \right )\mathcal{B}_{n-1}^{\alpha}\left ( z^{j+1};z \right )+z\sum_{i=0}^{n}\left ( \frac{i-1+1}{n-1} \right )^{j+1}F_{n-1,i-1}(z)  \right ]\\
&=&\left ( 1-\frac{1}{n} \right )^{j+1}\left [ \left ( 1-z \right )\mathcal{B}_{n-1}^{\alpha}\left ( z^{j+1};z \right )+z\sum_{k=0}^{j+1} \binom{j+1}{k}\dfrac{1}{\left ( n-1 \right )^{k}} \sum_{i=0}^{n-1}\left ( \frac{i}{n-1} \right )^{j+1-k}F_{n-1,i}(z) \right ]\\
&=&\left ( 1-\frac{1}{n} \right )^{j+1}\left [ \left ( 1-z \right )\mathcal{B}_{n-1}^{\alpha}\left ( z^{j+1};z \right )+z\sum_{k=0}^{j+1} \binom{j+1}{k}\dfrac{1}{\left ( n-1 \right )^{k}} \mathcal{B}_{n-1}^{\alpha}\left ( z^{j+1-k};z \right )\right ].
\end{eqnarray*}
We utilized the outcome of the following relation in the above calculations.
\begin{eqnarray*}
\sum_{i=0}^{n}\left( \frac{i-1+1}{n-1} \right)^{j+1}F_{n-1,i-1}(z) 
&=&\sum_{i=0}^{n}\dfrac{1}{\left ( n-1 \right )^{j+1}}\left [\sum_{k=0}^{j+1}\binom{j+1}{k}\left ( i-1 \right )^{j+1-k}  \right ]F_{n-1,i-1}(z) \\
&=&\sum_{i=0}^{n} \left[\sum_{k=0}^{j+1}\binom{j+1}{k}\dfrac{1}{\left( n-1 \right)^{k}}\left( \dfrac{i-1}{n-1} \right)^{j+1-k}  \right]F_{n-1,i-1}(z) \\
&=&\sum_{k=0}^{j+1} \binom{j+1}{k}\dfrac{1}{\left( n-1 \right)^{k}} \sum_{i=0}^{n-1}\left( \frac{i}{n-1} \right)^{j+1-k}F_{n-1,i}(z). 
\end{eqnarray*}
\end{proof}

Next, we  present a Grüss-Voronovskaja-type theorem for the $\alpha$-Bernstein operators, utilizing the methodology outlined in \cite{gal2014gr}.
\begin{theorem}
Let $f, h \in C^{2}[0, 1]$,  then for any $z \in [0, 1]$, we have 
\begin{equation}
\underset{n\rightarrow \infty}{lim}~~n \left [\mathcal{B}_n^{\alpha}(fh; z) -\mathcal{B}_n^{\alpha}(f; z)\mathcal{B}_n^{\alpha}(h; z)  \right ]= z(1-z){f}'(z){h}'(z),
\end{equation}
where $ \alpha\in [0,1]$.
\end{theorem}
\begin{proof}
It is straightforward to express
\begin{eqnarray*}
&&\mathcal{B}_n^{\alpha}(fh; z) -\mathcal{B}_n^{\alpha}(f; z)\mathcal{B}_n^{\alpha}(h; z)\\
&=&\mathcal{B}_n^{\alpha}(fh; z)-f\left ( z \right )h\left ( z \right )+f\left ( z \right )h\left ( z \right ) -\mathcal{B}_n^{\alpha}(f; z)\mathcal{B}_n^{\alpha}(h; z)+\mathcal{B}_n^{\alpha}(f; z)h\left ( z \right )-\mathcal{B}_n^{\alpha}(f; z)h\left ( z \right )\\
&=&\left [ \mathcal{B}_n^{\alpha}(fh; z)-f\left ( z \right )h\left ( z \right ) \right ]+\left [ h\left ( z \right ) \Big ( f(z)-\mathcal{B}_n^{\alpha}(f; z)\Big)  \right ]+\left [ \mathcal{B}_n^{\alpha}(f; z) \Big ( h(z)-\mathcal{B}_n^{\alpha}(h; z)\Big)  \right ].
\end{eqnarray*}
Considering corresponding Korovkin \cite{korovkin1953convergence} and Vornovskaja \cite{voronovskaja1932determination}  theorems from  the original paper \cite{chen2017approximation}, we derive the following
\begin{eqnarray*}
&&
\underset{n\rightarrow \infty}{lim} n \left[\mathcal{B}_n^{\alpha}(fh; z) -\mathcal{B}_n^{\alpha}(f; z)\mathcal{B}_n^{\alpha}(h; z)\right]\\
&=&
\underset{n\rightarrow \infty}{lim}n \left [ \mathcal{B}_n^{\alpha}(fh; z)-f\left ( z \right )h\left ( z \right ) \right ]+\left [ h\left ( z \right )
 \Big (\underset{n\rightarrow \infty}{lim} n [f(z)-\mathcal{B}_n^{\alpha}(f; z)]\Big)  \right ]\\
 &&~~~~+\left [ 
\underset{n\rightarrow \infty}{lim} \mathcal{B}_n^{\alpha}(f; z) 
 \Big (\underset{n\rightarrow \infty}{lim}n[ h(z)-\mathcal{B}_n^{\alpha}(h; z)]\Big)  \right ]\\
&=&\dfrac{1}{2}z(1- z)(fh)''(z)- \dfrac{1}{2}z(1- z){f}''(z)h\left ( z \right )-\dfrac{1}{2}z(1- z){h}''(z)f\left ( z \right )\\
&=&\dfrac{1}{2}z(1- z)\left [ (fh)''(z) -{f}''(z)h\left ( z \right ) -{h}''(z)f\left ( z \right )\right ]=z(1-z){f}'(z){h}'(z).
\end{eqnarray*}
\end{proof}

\subsection{Monotonicity preservation}

\begin{definition}
A system of functions $ \{h_{0},\cdots,h_{n}\}$ is monotonicity preserving 
 if for any $\nu _{0}\leq \nu _{1} \leq  \cdots\leq \nu_{n} $ 
in $ \mathbb{R} $, the function $ \sum_{i=0}^{n}\nu  _{i}h_{i} $ is increasing. 
\end{definition}
The Proposition 2.3 of \cite{carnicer1996convexity} outlines the characterization of systems that preserve monotonicity, as presented in the following result. 
\begin{prop}\label{pro1}
Let $\{h_{0},\cdots,h_{n}\}$ be a system of functions defined on an interval $ [a,b]$. Let $ k_{i}:=\sum_{j=i}^{n}h_{j}  $ \ for
$ i \in \left \{  0,1,\cdots,n\right \}$. Then $ \{h_{0},\cdots,h_{n}\}$ is monotonicity preserving if and only if $ k_{0} $ is a constant function and the functions $ k_{i} $ are
increasing for $ i = 1,\cdots,n$. 

\end{prop}
The subsequent result is derived from the preceding proposition.
\begin{prop}\label{pro3.2}
The $\alpha$-Bernstein polynomials $\{F_{n,0}, F_{n,1}, \cdots, F_{n,n}\}$, defined by equations (\ref{formula1}) and (\ref{rec-rel}),  are  monotonicity preserving.
\end{prop}
\begin{proof}
We use the mathematical  induction to verify the result. Consider the $\alpha$-Bernstein polynomials  defined by equation (\ref{rec-rel}) for $n=2$.  According to Proposition \ref{pro1}, the set $\{F_{2,0}, F_{2,1}, F_{2,2}\}$ preserves monotonicity under the conditions that $k_{0}$ is a constant function, and both $k_{1}$ and $k_{2}$ are increasing functions.

Utilizing the partition of unity of $\alpha$-Bernstein polynomials, we obtain
$k_{0}=\sum_{j=0}^{2}F_{2,j}(z)=1$.
On the other hand
\begin{equation*}
\sum_{j=1}^{2}F_{2,j}(z)=z-\alpha(z^{2}-z)\Rightarrow \dfrac{d}{dz} \Big(\sum_{j=1}^{2}F_{2,j}(z)\Big)=1-\alpha(2z-1)\geq 1-\alpha,
\end{equation*}
where $1-\alpha$ takes a non-negative value, therefore $\sum_{j=1}^{2}F_{2,j}(z)$ is an increasing function. Similarly, we show the monotonicity of the last function by taking derivative
\begin{equation*}
\sum_{j=2}^{2}F_{2,j}(z)=z+\alpha(z^{2}-z)\Rightarrow \dfrac{d}{dz} \Big(F_{2,2}(z)\Big)=1+\alpha(2z-1),
\end{equation*}
 where obviously $1+\alpha(2z-1) \geq 0$, so we have the desired result.

Now, we assume that  the $\alpha$-Bernstein polynomials  $\{F_{n-1,0}, F_{n-1,1}, \cdots, F_{n-1,n-1}\}$ are monotonicity preserving, and it remains to conclude the same result for any  $\alpha$-Bernstein polynomials of order $n$,  $\{F_{n,0}, F_{n,1}, \cdots, F_{n,n}\}$.

To do so, we first observe that $k_0=\sum_{j=0}^{n}F_{n,j}(z)=1$, i.e. is a constant function due to  partition of unity. 
For   $k_{i}(z)=\sum_{j=i}^{n}F_{n,j}(z)$, we show the monotonicity by verifying the non-negativity of its  derivative, 
 $ \dfrac{d}{dz}k_{i}(z)=\sum_{j=i}^{n} \dfrac{d}{dz}F_{n,j}(z)$.

Employing relation (\ref{rec-rel}), we have
\begin{eqnarray}
\dfrac{d}{dz}k_{i}(z)&=&-\sum_{j=i}^{n}F_{n-1,j}(z)+(1-z)\sum_{j=i}^{n}\dfrac{d}{dz} F_{n-1,j}(z)+\sum_{j=i}^{n}F_{n-1,j-1}(z)+z\sum_{j=i}^{n}\dfrac{d}{dz}F_{n-1,j-1}(z)\nonumber \\&=&F_{n-1,n}(z)+\left ( -\sum_{j=i}^{n-1}F_{n-1,j}(z)+\sum_{j=i}^{n-1}F_{n-1,j}(z) \right )+F_{n-1,i-1}(z)\nonumber \\&&+(1-z)\sum_{j=i}^{n}\dfrac{d}{dz} F_{n-1,j}(z)+z\sum_{j=i}^{n}\dfrac{d}{dz} F_{n-1,j-1}(z)\nonumber \\&=&F_{n-1,i-1}(z)+(1-z)\sum_{j=i}^{n-1}\dfrac{d}{dz} F_{n-1,j}(z)+z\sum_{j=i-1}^{n-1}\dfrac{d}{dz} F_{n-1,j}(z).\label{11}
\end{eqnarray}
Now, we use the induction hypothesis to see that
\begin{equation}
\sum_{j=i}^{n-1}\dfrac{d}{dz} F_{n-1,j}(z)\geq0, \hspace{1cm}\sum_{j=i-1}^{n-1}\dfrac{d}{dz} F_{n-1,j}(z)\geq0. \label{12}
\end{equation}
and  then it becomes evident that $ \dfrac{d}{dz} k _{i}(z)\geq0$.
\end{proof}

The monotonicity preservation of the $\alpha$-Bernstein operators follows directly from the preceding proposition.

\begin{theorem}
For a function $q$ that is continuous and monotonically increasing on the interval $[0,1]$, the corresponding $\alpha$-Bernstein operators are also increasing.
\end{theorem}
\begin{proof}

Considering the monotonically increasing function $q$, we observe the inequality $q\left(\frac{0}{n}\right) \leq q\left(\frac{1}{n}\right) \leq \cdots \leq q\left(\frac{n}{n}\right)$. Consequently, the result is directly derived from Proposition \ref{pro3.2}.

\end{proof}

\subsection{Convexity preservation}
The preservation of convexity by the $\alpha$-Bernstein operators is verified by the subsequent theorem.

\begin{theorem}
 If $q \in C[0,1] $ is a convex function, then    its $\alpha$-Bernstein operators are all convex.
 
\end{theorem}

\begin{proof}
We employ the method of induction to establish the result's validity, starting with the base case $n=2$.  Consider a set of real values $ \lambda_{0}, \lambda _{1} , \lambda _{2}$ forming a convex data set, ensuring $\lambda_{2} -2\lambda_{1} +\lambda_{0}\geq 0 $. Our objective is to demonstrate the convexity of the function $ G(z,\alpha)=\sum_{i=0}^{2}\lambda_{i}  F_{2,i}(z)$ over the interval $[0,1]$. Evaluating the second derivative, we get

\begin{equation*}
\dfrac{d^2}{dz^2} G(z,\alpha)=\sum_{i=0}^{2} \lambda_{i}  \dfrac{d^2}{dz^2}  F_{2,i}(z)= 2\alpha \left (\lambda_{2} -2\lambda_{1} +\lambda_{0}   \right )
\end{equation*}

Given this expression and $ \lambda_{2} -2\lambda_{1} +\lambda_{0}\geq 0 $, we can affirm that $ \dfrac{d^2}{dz^2} G(z,\alpha) \geq 0 $.

Presuming the convexity of a set of real values $ \{\beta_{i} \}_{i=0}^{n}$, we take as our induction hypothesis that the expression $ \sum_{i=0}^{n} \beta_{i} F_{n,i}(z) $ is convex.

Now, let's introduce a new set of convex values $ \{\eta_{i}\}_{i=0}^{n+1}$. The goal is to establish the convexity of the function $ \sum_{i=0}^{n+1} \eta_{i} F_{n+1,i}(z) $. To accomplish this, we consider
\begin{equation*} 
G(z,\alpha) =\sum_{i=0}^{n+1} \eta_{i} F_{n+1,i}(z)=\sum_{i=0}^{n+1} \eta_{i}\left [(1-z)F_{n,i}(z)+zF_{n,i-1}(z) \right ],  \end{equation*}
and establish that the function $ \dfrac{d}{dz}G(z,\alpha) $ is  increasing.
\begin{eqnarray*}
\dfrac{d}{dz} G(z,\alpha) &=&\sum_{i=0}^{n+1} \eta_{i} \dfrac{d}{dz} \left [(1-z)F_{n,i}(z)+zF_{n,i-1}(z)  \right ]\\
&=&\sum_{i=0}^{n+1} \eta_{i}\left [-F_{n,i}(z)+(1-z) \dfrac{d}{dz} F_{n,i}(z)+F_{n,i-1}(z)+z \dfrac{d}{dz} F_{n,i-1}(z)  \right ]\\
&=&\sum_{i=0}^{n} ( \eta_{i+1}-\eta_{i})F_{n,i}(z)+(1-z)\sum_{i=0}^{n}\eta_{i} \dfrac{d}{dz} F_{n,i}(z)+z \sum_{i=0}^{n}\eta_{i+1}\dfrac{d}{dz} F_{n,i}(z)
\end{eqnarray*}
Given that $\left \{ \eta_{i}  \right \}_{i=0}^{n+1}$ represents convex data, we observe that $  \eta_{i+1}-2 \eta_{i}+ \eta_{i-1}\geq 0, i=1,\cdots,n-1  $. This inequality implies $  \eta_{i+1}- \eta_{i}\geq \eta_{i}- \eta_{i-1} , i=1,\cdots,n-1  $. In simpler terms, the differences $ \eta_{i}- \eta_{i-1} , i=1,\cdots,n-1 $ form an increasing sequence. By applying Proposition \ref{pro3.2}, we can affirm that the function $ \sum_{i=0}^{n} ( \eta_{i+1}-\eta_{i})F_{n,i}(z) $ is an increasing function.

Assuming the induction hypothesis and given that $\left \{ \eta_{i} \right \}_{i=0}^{n+1}$ represents convex data, it follows that both $ \sum_{i=0}^{n}\eta_{i}F_{n,i}(z) $ and $ \sum_{i=0}^{n}\eta_{i+1}F_{n,i}(z) $ are convex functions. This implies that their respective derivatives, $ \sum_{i=0}^{n}\eta_{i} \dfrac{d}{dz} F_{n,i}(z)$ and $\sum_{i=0}^{n}\eta_{i+1} \dfrac{d}{dz} F_{n,i}(z) $, are increasing functions.

By the derived results and considering $ z \in [0,1] $, it follows that $ \dfrac{d}{dz} G(z,\alpha) $ is an increasing function. Consequently, this implies $ \dfrac{d^2}{dz^2} G(z,\alpha) \geq 0 $, signifying that $ G(z,\alpha) $ is a convex function.

We observe that if $ q $ is a convex function, then the values $ q\left(\frac{0}{n}\right), q\left(\frac{1}{n}\right), \ldots, q\left(\frac{n}{n}\right) $ form a set of convex data. Consequently, this directly implies that $ R_{n,\alpha}(q;z) $ is also a convex function.
\end{proof}

\section{Conclusion}

The current study makes two significant contributions to the field of approximation theory, particularly concerning Bernstein-like operators.

Firstly, a somewhat generalized formula is introduced for generating Bernstein-like functions. The study illustrates three specific families and opens avenues for constructing new and innovative families.

Secondly, it is demonstrated that the newly presented structure offers a recursive formula for generating $\alpha$-Bernstein polynomials. Utilizing this  representation, several properties are derived, and alternative proofs are provided for established results.





 \bibliographystyle{elsarticle-num} 
 \bibliography{ref}





\end{document}